\documentclass[a4paper]{amsart}
\usepackage{cmap}
\usepackage[T2A]{fontenc}
\usepackage[utf8]{inputenc}
\usepackage{amsfonts}
\usepackage{amsmath}
\usepackage{mathtools}
\usepackage{upgreek}
\usepackage{graphicx}
\usepackage{amssymb}
\usepackage{amsthm}
\usepackage{enumitem}
\usepackage{esint}
\usepackage{hyperref}
\usepackage{comment}
\theoremstyle{plain}
\newtheorem{theorem}{Theorem}[section]
\newtheorem{lemma}[theorem]{Lemma}
\newtheorem{prop}[theorem]{Proposition}

\newtheorem*{lemma*}{Lemma}
\newtheorem*{task*}{Problem}
\newtheorem*{theorem*}{Theorem}
\theoremstyle{definition}

\newtheorem{remark}[theorem]{Remark}

\begin{document}
\title[On projective $3$-folds with two-dimensional space of vanishing
cycles]{On projective threefolds with two-dimensional space of vanishing
cycles}
\author{Timofey Fedorov}
\address{National Research University Higher School of Economics, Moscow, \mbox{Russia}.}
\email{tsfedorov@edu.hse.ru}
\begin{abstract}
We obtain a complete list of smooth projective threefolds
over~$\mathbb C$ for which the dimension of the space of vanishing
cycles (in $H^2$ of the smooth hyperplane section) equals~$2$.
We also obtain a complete list of rank $2$ very ample vector bundles $E$ on smooth projective surfaces with $c_2(E) = 3$.
\end{abstract}
\maketitle
\section{Introduction}
Suppose that $X\subset \mathbb P^n$ is a smooth projective threefold
over~$\mathbb C$ and that $Y\subset X$ is its smooth hyperplane
section. Put $\mathrm{ev}(Y)=b_2(Y)-b_2(X)\ge 0$, where $b_2$ stands
for the second Betti number. The number $\mathrm{ev}(Y)$ is the
dimension of the space of vanishing cycles. According to \cite[Expos\'e XIX]{sga}, $\mathrm{ev}(Y)=0$ if and only if the dual of the threefold $X$ is not a hypersurface, and the classification of such threefolds is easily derived from \cite[Corollary 3.2]{Ein}
(it suffices to observe that, since $\dim Y$ is even, vanishing cycles are not homologous to zero, and that Lefschetz pencils do not contain singular fibers if and only if $X^*$ is not a hypersurface).
 There also exists a complete
classification of threefolds for which $\mathrm{ev}(Y)=1$
(see~\cite[Theorem 1.1]{Lvovski}). In this paper we treat the
next case $\mathrm{ev}(Y)=2$. It turns out that such varieties admit a
full classification as well.

We will say that $X$ is a scroll over a surface if
    there exists a surface $S$ with a very ample rank-$2$ bundle $E$ such that $X \cong \mathbb P(E)$ is embedded in $\mathbb P^n$ via tautological line bundle $\mathcal O_{\mathbb P(E)}(1)$;
    $X$ is a pencil of quadrics (also called hyperquadric fibration) if there exists a morphism $p\colon X \to C$,
    where $C$ is a smooth curve, such that the fiber of $p$ over a general
    point of $C$ is a smooth quadric (i.e., a smooth surface of degree $2$
    in $\mathbb P^n$).
The main result of the paper is
as follows.

\begin{theorem}
Suppose that $X\subset \mathbb P^n$ is a smooth projective threefold
over~$\mathbb C$ and that $Y\subset X$ is its smooth hyperplane
section. Then $\mathrm{ev}(Y)=2$ if and only if $X$ is one of the
following varieties,

\begin{enumerate}
    \item $X = \mathbb P_S(E)$ is a scroll over a surface and $(S, E)$ is one of the following:
    \begin{enumerate}
        \item $(S, E) \cong (\mathbb P^2, \mathcal O(1)\oplus \mathcal O(3))$;
        \item $(S, E) \cong (\mathbb P^2, T_{\mathbb P^2})$;
        \item $(S, E) \cong (\mathbb P^1 \times \mathbb P^1, \mathcal O_{\mathbb P^1 \times \mathbb P^1}(1,1)\oplus  \mathcal O_{\mathbb P^1 \times \mathbb P^1}(1,2))$;
        \item $(S, E) \cong (\mathbb F_1, [C_0 + 2f]^{\oplus 2})$;
        \item $(S, E) \cong (S, [-K_S]^{\oplus 2})$, where $S$ is a smooth cubic in $\mathbb P^3$;
    \end{enumerate}
    \item $X$ is a pencil of quadrics. Let $p\colon X \to C$ be the corresponding morphism. Then $C \cong \mathbb P^1$, $X$ is a divisor in the projective bundle $\mathcal W \coloneqq \mathbb P(\mathcal O_{\mathbb P^1}(2) \oplus \mathcal O_{\mathbb P^1}(2) \oplus \mathcal O_{\mathbb P^1}(1) \oplus \mathcal O_{\mathbb P^1}(1))$
    and $X$ is embedded via $i^*\mathcal O_{\mathcal W} (1)$ ($i \colon X \hookrightarrow W$ denotes the inclusion); all fibers of $p$ are smooth.
    We can associate to each such $X$ a global section of the vector bundle
    \begin{equation*}
        \begin{pmatrix}
            \mathcal O(-1)&\mathcal O(-1)&\mathcal O    &\mathcal O    \\
            \mathcal O(-1)&\mathcal O(-1)&\mathcal O    &\mathcal O    \\
            \mathcal O   &\mathcal O   &\mathcal O(1)&\mathcal O(1)\\
            \mathcal O   &\mathcal O   &\mathcal O(1)&\mathcal O(1)\\
        \end{pmatrix}
    \end{equation*}
    of rank $16$ such that its determinant is non-zero at any point $a \in \mathbb P^1$.
\end{enumerate}
\end{theorem}

The proof is based on the following observation (see Lemma~\ref{finite_monodromy_lemma}
below): if $\mathrm{ev}(Y)=2$, then the monodromy group acting on
$H^2(Y,\mathbb R)$ is finite (and isomorphic to the symmetric
group~$S_3$). Using results of the paper~\cite{Lvovski}, where the
list of threefolds with finite monodromy group is read off the list of
threefolds with empty adjoint system from Sommese's
paper~\cite{Sommese}, and where monodromy groups are found for (almost
all) such varieties, as well as the classification results from the
paper \cite{Takahashi},
we extract from this list the varieties for which the monodromy group
is~$S_3$, which yields the result.

The paper is organized as follows. In Section~\ref{finite_monodromy} we prove that
the equation $\mathrm{ev}(Y)=2$ implies that the monodromy group is
isomorphic to $S_3$. We also show that, in the latter
case, either $(X,\mathcal O_X(1))\cong (\mathbb P(E),\mathcal O_{\mathbb P(E)|S}(1))$,
where $E$ is a very ample bundle of rank~$2$ over a
surface~$S$, or $X$ is a pencil of quadrics over a smooth curve. The
case of $\mathbb P^1$-scrolls is treated in Section~\ref{scrolls}, the
case of pencils of quadrics is treated in Section~\ref{pencils_of_quadrics}.

The study was implemented in the framework of the Basic Research Program at the National Research University Higher School of Economics (HSE University) in 2025.

I am grateful to my advisor Serge Lvovski for support and
guidance.

\section{Finite monodromy}
\label{finite_monodromy}
Let us define the space $\mathrm{Ev}(Y)\coloneqq (i^*H^2(X,\mathbb C))^\perp$, where $i^*\colon H^2(X,\mathbb C)\to H^2(Y,\mathbb C)$ is the restriction mapping and the orthogonal complement is taken with respect to the cup-product.
We will call $\mathrm{Ev}(Y)$ the space of vanishing cycles.
\begin{lemma}
\label{finite_monodromy_lemma}
Suppose a threefold $X\subset \mathbb P^n$ has $2$-dimensional space of vanishing cycles.
Then $\mathrm{Ev}(Y) \subset H^{1,1}(Y) \cap H^2(Y, \mathbb Q)$, and the monodromy group $G \subset \mathrm{GL}(\mathrm{Ev}(Y))$ is isomorphic to $S_3$.
\end{lemma}
\begin{proof}
    According to \cite[Expos\'e XIX]{sga}, the monodromy group acting on $H^2(Y,\mathbb R)$ is finite if and only if $\mathrm{Ev}(Y)$ is contained in $H^{1,1}(Y)$, and if it is finite and $\mathrm{ev}(Y)=r >0$,
    then this group is isomorphic to the Weyl group of a rank $r$  simply laced irreducible root system.
    It is clear that $\mathrm{Ev}(Y)=\bigoplus_{p+q=2}(\mathrm{Ev}(Y)\cap H^{p,q}(Y))$.
    If, arguing by contradiction, $\mathrm{ev}(Y)=2$ but $\mathrm{Ev}(Y)$ is not contained in $H^{1,1}(Y)$,
    then $\mathrm{Ev}(Y)\subset H^{2,0}(Y)\oplus H^{0,2}(Y)$, whence $\int _Y(\omega\wedge\omega)\geq0$ for any $2$-form $\omega$ representing an element of $\mathrm{Ev}(Y)\cap H^2(Y,\mathbb R)$.
    Since the latter space is spanned by fundamental classes of vanishing cycles, which have self-intersection index~$-2$, we arrive at a contradiction.
\end{proof}
According to Proposition 4.2 in \cite{Lvovski}, threefolds with finite and non-trivial monodromy group are as follows:
\begin{prop}[A. Sommese, S. Lvovski]
    \label{sommese-lvovski}
    Suppose that $X \subset \mathbb P^N$ is a smooth threefold such that $X^*$
 is a hypersurface (equivalently, the monodromy group of $X$ is non-trivial).
 Then the monodromy group of $X$ is finite if and only if $X$ is one of the
 varieties listed below:
    \begin{enumerate}
        \item $X$ is a scroll over a surface, that is, there exists a locally free
        sheaf $E$ of rank $2$ over a smooth surface $S$ such that $(X,\mathcal O_X(1)) \cong (\mathbb P(E),\mathcal O_{\mathbb P(E)}(1))$
        \item $X$ is a pencil of quadrics, that is, there exists a morphism $p\colon X \to C$,
        where $C$ is a smooth curve, such that the fiber of $p$ over a general
        point of $C$ is a smooth quadric (i.e., a smooth surface of degree $2$
        in $\mathbb P^N$).
        \item $X$ is a Veronese pencil, that is, there exists a morphism $p\colon X \to C$, where $C$ is a smooth curve, such that, for a general point $a \in C$,
        the fiber $X_a = p^{-1}(a)$ is a smooth surface and $(X_a,\mathcal O_{X_a}(1)) \cong (\mathbb P^2, \mathcal O(2))$
        \item $X$ is a Del Pezzo threefold, i.e., a Fano variety embedded by one
        half of the anticanonical class, that is, $\omega_X \cong \mathcal O_X(-2)$
        \item $X$ is the smooth quadric in $\mathbb P^4$
        \item $X$ is the Veronese image $v_2(Q) \subset \mathbb P^{13}$ or its isomorphic projection.
        \item $X\subset \mathbb P^n$ is the blowup of the
        smooth three-dimensional quadric $Q$ at $k \geq 1$ points, and $\mathcal O_X(1) \cong \mathcal O_X(2\sigma^*H - E_1 - \dots - E_k)$,
        where $\sigma \colon X \to Q$ is the blowdown
        morphism, $H$ is a hyperplane section of $Q$, and $E_1,\dots ,E_k \subset X$ are
        exceptional divisors.
    \end{enumerate}
\end{prop}
It follows from \cite{Lvovski} that in the classes 3-7 only one variety has two-dimensional space of vanishing cycles: this is the Del Pezzo threefold $\mathbb P(E)$, $E = T_{\mathbb P^2}$, embedded by the complete linear system $|\mathcal O_{\mathbb P(E)}(1)|$ (this variety belongs to the class~1 as well).

Hence, all threefolds that we are interested in belong to the classes $1$ and $2$.

\section{Threefold scrolls over surfaces}
\label{scrolls}
\subsection{Reduction to Takahashi's classification}
\label{reduction_takahashi}
Let us fix notation. Let $S$ be a smooth projective surface, $E$ be a rank $2$ very ample bundle, $X$ be the corresponding projective bundle $\mathbb P(E)$ with tautological line bundle $\mathcal O(1)$ and $\pi \colon \mathbb P(E) \to S$ be the projection.
Put $n =\mathrm {dim}\, \mathbb P(H^0(S, E)) = \mathrm {dim}\, \mathbb P(H^0(X, \mathcal O(1)))$
so $X$ is embedded in $\mathbb P(H^0(X, \mathcal O(1))) = \mathbb P^n$ via $\mathcal O(1)$; denote by $s_y$ the fiber of projection $\pi$ above the point $y \in S$
(note that the fibers $s_y, y \in S$ are actually lines in $\mathbb P^n$). Denote by $H_{\alpha}$ the hyperplane section corresponding to the (non-zero) section $\alpha \in H^0(E)$.
We can choose a smooth hyperplane section $Y \subset X$ containing exactly $c_2(E)$ different fibers of projection $\pi$.
Then $Y$ is isomorphic to the blow-up of $S$ at $c_2(E)$ different points.
Since $b_2(X) = b_2(Y)+1$, we have $\mathrm{ev}(Y) = c_2(E) - 1$, so $c_2(E) = 3$.
The problem of classifying very ample and, more generally, ample rank-$2$ bundles $E$ on surfaces with small $c_2(E)$ was considered by many authors; the paper \cite{Noma1994} contains classification of all such pairs with ample and globally generated $E$ and $c_2(E)\leq 2$.
We use the classification result of Takahashi \cite{Takahashi} to extend classification of very ample rank 2 bundles on surfaces to the case $c_2(E) = 3$.

We need one definition to present the main result of \cite{Takahashi}. A line bundle $L$ on a scheme $X$ is called $k$-jet ample if for any choice of distinct points $x_1, \dots, x_r$ and positive integers $k_1, \dots, k_r$ with $\sum_{i=1}^r k_i = k+1$ the natural map
\[
H^0(L) \to H^0\left(L \otimes \left(\frac{\mathcal O_{x_1}}{\mathfrak{m}_{x_1}^{k_1}} \oplus \dots \oplus \frac{\mathcal O_{x_r}}{\mathfrak{m}_{x_r}^{k_r}}\right)\right)
\]
is surjective (note that \textit{$1$-jet ampleness} is equivalent to \textit{very ampleness}).
It is proved in \cite{Takahashi} that for an ample rank $2$ bundle $E$ with $k$-jet ample determinant bundle $\mathrm{det}\, E$ the inequality $c_2(E) \geq k-1$ holds;
a complete list of rank $2$ ample bundles with $k$-jet ample determinant satisfying $k-1 \leq c_2(E) \leq k+1$ is obtained in this paper (\cite[Theorem 1.1]{Takahashi}).
Theorem \ref{2-jet} says that for a very ample rank $2$ bundle $E$ on a surface its determinant $\mathrm{det}(E)$ is $2$-jet ample.
Our classification of very ample bundles with $c_2(E) = 3$ will follow from \cite[Theorem 1.1]{Takahashi}, once one has extracted the very ample bundles from this list.
Here is the part of this theorem:
\begin{theorem*}[\cite{Takahashi}]
    \label{Takahashi}
    Let $S$ be a smooth connected projective surface and $E$ an ample
    and spanned vector bundle of rank $2$ on $S$. Assume that $\mathrm{det}\, E$ is $k$-jet ample for
    $k \geq 1$. Then $c_2(E) = k+1$ if and only if $(S, E)$ is one of the following:
        \begin{enumerate}
            \item $(S, E)\cong (\mathbb P^2, \mathcal O_{\mathbb P^2}\oplus \mathcal O_{\mathbb P^2}(k+1))$;
            \item $k=2$ and $(S, E) \cong (\mathbb P^2, T_{\mathbb P^2})$;
            \item $k = 3$ and $(S, E)\cong (\mathbb P^2, \mathcal O_{\mathbb P^2}(2)^{\oplus 2})$;
            \item $k=4$, $S\cong \mathbb P^2$, $E$ is semistable, but not stable, and there is an exact sequence $0 \to \mathcal O_{\mathbb P^2}(2)\to E \to \mathcal I_x(1)\to 0$, where $x$ is a point of $\mathbb P^2$ and $\mathcal I_x$ is the ideal sheaf of the $0$-dimensional subscheme $\{x\}$;
            \item $k = 5$ and $(S, E) \cong (\mathbb P^2, \mathcal O_{\mathbb P^2}(2)\oplus \mathcal O_{\mathbb P^2}(3))$;
            \item $k = 1$ and $(S, E) \cong (\mathbb P^1 \times \mathbb P^1, \mathcal O_{\mathbb P^1 \times \mathbb P^1}(1,1)^{\oplus 2} )$;
            \item $k = 2$ and $(S, E) \cong (\mathbb P^1 \times \mathbb P^1, \mathcal O_{\mathbb P^1 \times \mathbb P^1}(1,1)\oplus  \mathcal O_{\mathbb P^1 \times \mathbb P^1}(1,2))$;
            \item $k = 3$ and $(S, E) \cong (\mathbb P^1 \times \mathbb P^1, \mathcal O_{\mathbb P^1 \times \mathbb P^1}(1,1)\oplus  \mathcal O_{\mathbb P^1 \times \mathbb P^1}(2,2))$;
            \item $k = 2$ and $(S, E) \cong (\mathbb F_1, [C_0 + 2f]^{\oplus 2})$;
            \item $k = 1$ and $(S, E) \cong (Bl_7(\mathbb P^2), [-K_S]^{\oplus 2})$ (here and below $\mathrm{Bl}_j(S)$ is the surface obtained by blowing up $S$ at $j$
            points in general position);
            \item $k = 2$ and $(S, E) \cong (Bl_6(\mathbb P^2), [-K_S]^{\oplus 2})$;
            \item $k = 1$, $p \colon S \to C$ is a $\mathbb P^1$-bundle over an elliptic curve $C$ with invariant $e = -1$, and $E \cong p^*(\mathcal E)\otimes [C_0]$, where $\mathcal E$ is an indecomposable rank-$2$ vector bundle on $C$ of degree $1$;
            \item $k = 2$, $S$ is a $\mathbb P^1$-bundle over an elliptic curve $C$ with invariant $e = -1$, and $E \cong [C_0 + f]^{\oplus 2}$;
            \item $k = 2$, $p \colon S \to C$ is a $\mathbb P^1$-bundle over an elliptic curve $C$ with invariant $e = -1$, and $E \cong p^*(\mathcal E)\otimes [C_0]$, where $\mathcal E$ is an indecomposable rank-$2$ vector bundle on $C$ of degree $2$;
            \item $k = 2$, $S$ is a $K3$ surface, and $c_1(E)^2 = 10$, or $12$;
            \item $k = 2$, $S$ is an Enriques surface, and $c_1(E)^2 = 12$;
        \end{enumerate}
    Here $C_0$ is a section of minimal self-intersection $C_0^2 = -e$, and $f$ is a fiber of the ruling.
\end{theorem*}
We will refer to this list as Takahashi's list.

We do not know whether \textit{ample} bundles in the cases $(15)$ and $(16)$ exist or not but we show in this paper that these bundles cannot be \textit{very ample} (see Section \ref{no_k3}).
Theorem \ref{2-jet} says that for a very ample rank $2$ bundle $E$ on a surface its determinant $\mathrm{det}(E)$ is $2$-jet ample.
Our classification of very ample bundles with $c_2(E) = 3$ follows from Takahashi's result, one needs only to extract very ample bundles from the list.
To that end, we have
\begin{theorem}
    \label{classification3}
    Let $E$ be a very ample rank-$2$ bundle on a smooth projective surface $S$ with $c_2(E) = 3$ over the complex number field.
    Then $(S, E)$ is one of the following:
    \begin{enumerate}
        \item $(S, E) \cong (\mathbb P^2, \mathcal O(1)\oplus \mathcal O(3))$;
        \item $(S, E) \cong (\mathbb P^2, T_{\mathbb P^2})$;
        \item $(S, E) \cong (\mathbb P^1 \times \mathbb P^1, \mathcal O_{\mathbb P^1 \times \mathbb P^1}(1,1)\oplus  \mathcal O_{\mathbb P^1 \times \mathbb P^1}(1,2))$;
        \item $(S, E) \cong (\mathbb F_1, [C_0 + 2f]^{\oplus 2})$;
        \item $(S, E) \cong (S, [-K_S]^{\oplus 2})$, where $S$ is a smooth cubic in $\mathbb P^3$;
    \end{enumerate}
\end{theorem}

(The pairs $(S, E)$ with $E$ very ample and $c_2(E) < 3$ are listed in \cite[Theorem 11.1.3 and Theorem 11.4.5]{adjunction_theory}.)

\subsection{$2$-jet ampleness of the determinant bundle}
\label{2jetsection}
In this subsection we work over an arbitrary algebraically closed field $\Bbbk$.
\begin{theorem}
    \label{2-jet}
    Suppose that $E$ is a very ample rank-$2$ bundle on a surface $S$.
    Then $\mathrm{det}\, E$ is $2$-jet ample.
\end{theorem}
Before proving this statement we state a very simple lemma.
\begin{lemma}
    \label{obvious}
    Let $S$ and $E$ be as before. Take pairwise disctinct points $y, y_1,\dots, y_k$ and an arbitrary point $p \in s_{y_1}$.
    Then there is a hyperplane $H \subset \mathbb P^n$ containing $s_y$, not containing $s_{y_2}, \dots, s_{y_k}$ (hence, intersecting each of these fibers at one point) and intersecting the fiber $s_{y_1}$ at the point $p$.\qed
\end{lemma}

\begin{proof}{(of Theorem \ref{2-jet})}

    Consider $r\le3$ distinct points $y_1,\dots,y_r\in S$ and $r$ natural numbers $k_1,\dots,k_r$, $\sum k_i=3$.
    We need to examine three different cases:

    (1) $r=3$, $k_1=k_2=k_3=1$. We prove surjectivity of
    \[
    \varphi \colon H^0(\mathrm{det}\, E) \to H^0(\mathrm{det} \,E\otimes (\mathcal O_{y_1}/ \mathfrak m_{y_1}\oplus \mathcal O_{y_2}/ \mathfrak m_{y_2} \oplus \mathcal O_{y_3}/ \mathfrak m_{y_3})) \cong \Bbbk^3
    \]
    as follows. Choose a section $\alpha_1 \in H^0(E)$ such that $H_{\alpha_1} \supset s_{y_1}$ and $H_{\alpha_1}$ intersects each fiber $s_{y_2}, s_{y_3}$ at one point (using lemma \ref{obvious}); let $p$ be the intersection point of $H_{\alpha_1}$ and $s_{y_3}$.
    Choose a section $\alpha_2 \in H^0(E)$ such that $H_{\alpha_2} \supset s_{y_2}$, $H_{\alpha_2}$ intersects each fiber $s_{y_1}, s_{y_3}$ at one point and intersection of $H_{\alpha_2}$ with $s_{y_3}$ is not the point $p$.
    Then $\varphi (\alpha_1 \wedge \alpha_2)$ is equal to $(0, 0, \lambda), \lambda \neq 0$. Hence all standard basis elements are in the image of $\varphi$ and we are done.

    (2) $r=2$, $k_1=2$, $k_2=1$. We need to prove surjectivity of
    \[
    \varphi \colon H^0(\mathrm{det}\, E) \to H^0(\mathrm{det} \,E \otimes (\mathcal O_{y_1}/\mathfrak m_{y_1}^2\oplus \mathcal O_{y_2} / \mathfrak m_{y_2})) \cong \Bbbk^3\oplus \Bbbk
    \]
    Fix an affine neighborhood $U \ni y_1$ such that $E$ is trivial over $U$. Choose local coordinates $x_1, x_2$ in $\mathcal O_{S, y_1}$.
    Every $f\in H^0(U, E)$ is represented as the pair $f = (f_1, f_2), f_i \in \mathcal O_{S,y_1}$ and since $O_{S,y_1}$ is a regular local ring the images of $f_1, f_2$ in the completion $\widehat{\mathcal O_{S, y_1}} \cong \Bbbk[[x_1, x_2]]$ are represented by the power series $f_1(x_1, x_2), f_2(x_1, x_2)$ in variables $x_1, x_2$.
    Very ampleness of $\mathcal O(1)$ means, in particular, that $\mathcal O(1)$ generates $1$-jets in all points $p \in s_{y_1}$, i.e., for any $p \in s_{y_1}$ the natural map
    \[
    \psi_p \colon H^0(\mathcal O(1)) \to H^0(\mathcal O(1) \otimes \mathcal O_{X, p}/ \mathfrak m_{X, p}^2)
    \]
    is surjective.
    We want to rewrite these conditions for $\psi_p,\, p \in s_{y_1}$ in terms of sections $f\in H^0(E)$.
    Let $(t_1 \colon t_2)$ be homogeneous coordinates on $\mathbb P^1$, so $((t_1 \colon t_2), x_1, x_2)$ are coordinates for $\pi^{-1}(U) \cong \mathbb P^1 \times U$.
    Section $f\in H^0(E)$ written locally as $(f_1(x_1, x_2), f_2(x_1, x_2))$ correspond to the section $\bar f\in H^0(\mathrm{det}\, E)$ written as
    \[
    ((t_1:t_2), x_1, x_2) \mapsto t_1 f_2(x) - t_2 f_1(x),\, x = (x_1, x_2) \in U
    \]
    Consider sections $f \in H^0(E)$ such that $(f_1(y_1)\colon f_2(y_1)) = (1:\lambda), \lambda \in \Bbbk$ and take $u = t_2/t_1$ to be the affine coordinate near $(1:\lambda) \in \mathbb P^1$.
    In coordinates $(u, x_1, x_2)$ section $\bar f$ is written as
    \[
    (u, x_1, x_2) \mapsto f_2(x)-uf_1(x)
    \]
    The gradient vector of $\bar f$ w.r.t $u, x_1, x_2$ equals
    \[
    \left(-f_1(x), \frac{\partial f_2}{\partial x_1} - u\frac{\partial f_1}{\partial x_1}, \frac{\partial f_2}{\partial x_2} - u\frac{\partial f_1}{\partial x_2}\right)
    \]
    Therefore, we can restate the surjectivity condition of all $\psi_p$ as follows:
    \begin{prop}
        \label{matrices}
        Let $\psi_p, p \in s_{y_1}$ be the natural maps as above.
        All the maps $\psi_p$ are surjective iff the following holds: differentials of sections $f\in H^0(E)$ at the point $y_1$ form a linear subspace $\Pi\subset \mathfrak m_{y_1}/\mathfrak m_{y_1}^2 \oplus \mathfrak m_{y_1}/\mathfrak m_{y_1}^2 \cong \mathrm{Mat}_{2 \times 2}(\Bbbk)$
        such that for any nonzero vector $v = (t_1, t_2)^T \in \Bbbk^2$ there exist two matrices $A_1, A_2 \in \Pi$, for which vectors $A_1 v$ and $A_2 v$ are linearly independent,
        and for any $w \in \Bbbk^2$ and $A\in \Pi$ one can choose a section $f \in H^0(E)$ satisfying $f(y_1) = w, (df)_{y_1} = A$.\qed
    \end{prop}
    
    To complete the proof in the case (2) choose a section $g \in H^0(E)$ such that $g(y_1) = (1, 0)$ and the corresponding hyperplane contains fiber $s_{y_2}$ (i.e. $g(y_2) = 0$) using Lemma \ref{obvious}.
    Consider sections $f \in H^0(E)$ with $f(y_1) = 0$. Write power series representations of $g$ and $f$ in a $2$ by $2$ matrix $J$:
    \begin{equation*}
        \begin{pmatrix}
            1+g_1(x_1,x_2)&g_2(x_1,x_2)\\
            f_1(x_1,x_2)  &f_2(x_1,x_2)\\
        \end{pmatrix}
    \end{equation*}
    Here the power series $g_i, f_i$ are in $\widehat{m_{y_1}} \subset \widehat{\mathcal O_{S, y_1}}$.
    The differential of $g\wedge f$ at the point $y_1$ is the linear part of the power series $\mathrm{det}\, J$, so it is actually equal to the linear part of $f_2$.
    Due to Proposition \ref{matrices} we can choose $f$ such that section $g\wedge f$ have an arbitrary differential at $y_1$.
    Therefore the image of $\varphi$ contains the subspace $(\mathfrak m_{y_1}/\mathfrak m_{y_1}^2, 0) \subset \mathcal O_{y_1} / \mathfrak m_{y_1}^2 \oplus \mathcal O_{y_2}/ \mathfrak m_{y_2}$.
    Similarly to the case (1) one can a get vector of the form $(1+ g(x_1, x_2), 0)$ and the vector $(0, 1)$, hence $\varphi$ is surjective.
    
    (3) $r=1$, $k_1=3$; we will write $y$ instead of $y_1$. We are to prove the surjectivity of
    \[
    \varphi \colon H^0(\mathrm{det}\, E) \to H^0(\mathrm{det} \,E/ \mathfrak m_{y}^3)
    \]
    Let us analyse possible subspaces $\Pi \subset \mathrm{Mat}_{2 \times 2}(\Bbbk)$ satisfying conditions of Proposition \ref{matrices}.
    Note that if $\Pi$ is equal to $\mathrm{Mat}_{2 \times 2}(\Bbbk)$ there is nothing to prove: if one consider sections $f, g \in H^0(E)$ having zero at $y$, then quadratic part of $f\wedge g$ spans $\mathfrak m_y^2/\mathfrak m_y^3$.
    Combining this with arguments from the case (2) yields surjectivity of $\varphi$.
    
    Note also that $\Pi$ cannot be $2$-dimensional.
    Indeed, if $\Pi$ is spanned by $A_1, A_2 \in \mathrm{Mat}_{2 \times 2}(\Bbbk)$, then the hypothesis of Proposition \ref{matrices} implies that for any nonzero vector $v = (t_1, t_2)^T$ vectors $A_1 v, A_2 v$ are linearly independent.
    The determinant of the matrix constructed from two vectors $A_1 v, A_2 v$ is equal to
    \[
    t_1^2(a_1 c_2-a_2 c_1) +t_1 t_2 (a_1 d_2 - b_2 c_1 + b_1 c_2 - a_2 d_1) +t_2^2(b_1 d_2 - b_2 d_1)
    \]
    Clearly there exist a pair $(t_1, t_2)$ such that this expression vanishes, thus $\Pi$ cannot be of dimension two.

    Now consider the case of $3$-dimensional $\Pi$. We want to choose a suitable basis for $\Pi$.
    By Proposition \ref{matrices} there are two matrices $A_1, A_2$ in $\Pi$ such that $A_1 (1, 0)^T, A_2 (1, 0)^T$ are linearly independent. Hence we can choose $A_1, A_2$ of the form
    {\sloppy

    }
    \begin{equation*}
        A_1 =\begin{pmatrix}
            1&a_1\\
            0&b_1\\
        \end{pmatrix},\,
        A_2 =\begin{pmatrix}
            0&a_2\\
            1&b_2\\
        \end{pmatrix}
    \end{equation*}
    There exists a matrix $A_3 \in \Pi \setminus \mathrm{Span}\, (A_1, A_2)$ of the form
    \begin{equation*}
        A_3 =\begin{pmatrix}
            a_3&1\\
            b_3&0\\
        \end{pmatrix}
    \end{equation*}
    or
    \begin{equation*}
        A_3 =\begin{pmatrix}
            a_3&0\\
            b_3&1\\
        \end{pmatrix}
    \end{equation*}
    and after swapping the rows in all the matrices in $\Pi$ if necessary we can assume that $A_3$ has the form
    \begin{equation*}
        A_3 =\begin{pmatrix}
            a_3&1\\
            b_3&0\\
        \end{pmatrix}
    \end{equation*}
    We want to prove that sections of the form $f\wedge g \in H^0(\mathrm{det}\, E)$, $f(y) = g(y) = 0$, span $\mathfrak m_y^2/\mathfrak m_y^3$,
    that is, that the three homogeneous quadratic polynomials 
    \begin{align*}
    (A_1 \cdot (t_1, t_2)^T) &\wedge (A_2 \cdot (t_1, t_2)^T)\\
    (A_2 \cdot (t_1, t_2)^T) &\wedge (A_3 \cdot (t_1, t_2)^T)\\
    (A_1 \cdot (t_1, t_2)^T) &\wedge (A_3 \cdot (t_1, t_2)^T)\\
    \end{align*}
    in variables $t_1, t_2$ are linearly independent.
    We show this by direct computation using Maple.
    Writing these three polynomials into the $3$-by-$3$ matrix (the first polynomial in the first row and so on) with respect to the basis $t_1^2, t_1 t_2, t_2^2$ we have:
    \begin{equation*}
        D =\begin{pmatrix}
            1&a_1 + b_2 &a_1 b_2 - a_2 b_1  \\
            b_3&a_1 b_3 - b_1 a_3&-b_1      \\
            -a_3& a_2 b_3 - a_3 b_2 -1&-b_2 \\
        \end{pmatrix}
    \end{equation*}
    Note that $P_1$ does not vanish at the point $(t_1\colon t_2) = (1\colon 0)$. Putting $t_2 = 1$ we get polynomials in variable $t_1$:
    \begin{align*}
        P_1 &= t_1^2 + t_1(a_1 +b_2) + a_1 b_2 - a_2 b_1\\
        P_2 &= t_1^3 b_3 + t_1 (a_1 b_3 - b_1 a_3) - b_1\\
        P_3 &= t_1^3 (-a_3) +t_1(a_2 b_3 - a_3 b_2 -1) - b_2\\
    \end{align*}
    We have
    \begin{align*}
        \mathrm{det}\, D &= a_1^2 a_3 b_2 b_3 - a_1 a_2 a_2 b_1 b_3 + a_1 a_2 b_2 b_3^2 - a_1 a_3^2 b_1 b_2 - a_1 a_3 b_2^2 b_3 -\\
        &- a_2^2 b_1 b_3^2+a_2 a_3^2 b_1^2 + a_2 a_3 b_1 b_2 b_3 + a_1 a_3 b_1 -\\
        &-a_1 b_2 b_3 + 2 a_2 b_1 b_3 + a_3 b_1 b_2 + b_2^2 b_3 - b_1\\
    \end{align*}
    and $res(P_1, P_2) = -b_1 \cdot \mathrm{det} \, D$, $res(P_1, P_3) = a_2 \cdot \mathrm{det}\, D$. 
    Suppose now that $\mathrm{det}\, D$ is zero. Then $res(P_1, P_2) = res(P_1, P_3) = 0$.
    Since the rows of the matrix $D$ are linearly dependent, either the second row is a linear combination of the first and the third rows or the third row is a linear combination of the first and the second row.
    In both cases all three polynomials have a common zero contradicting Proposition \ref{matrices}.
    Thus $\mathrm{det}\, D \neq 0$ and we are done.
\end{proof}

\begin{remark}
    The bound for $k$-jet ampleness of $\mathrm{det}\, E$ is sharp: there are many examples of rank $2$ very ample bundles such that $\mathrm{det}\, E$ is not $3$-jet ample
    (in fact, even not $3$-very ample; $k$-very ampleness is weaker than $k$-jet ampleness, see \cite{Beltrametti_kjet}).
    Consider an arbitrary smooth surface $S \subset \mathbb P^3$ of degree at least $4$ and the very ample bundle $E =\mathcal O(1)\oplus \mathcal O(1)$ on $S$.
    Then $\mathrm{det}\, E \cong \mathcal O(2)$ does not separate any four different points of $S$,
    lying on one line in $\mathbb P^3$
    (if section from $H^0(S, \mathcal O(2)) \cong H^0(\mathbb P^3, \mathcal O(2))$ vanishes at three collinear points then this section vanishes on the whole line through this points).
    For cubics and quadrics the bundle $\mathrm{det}\, E$ is not $3$-jet ample, either.
\end{remark}
\subsection{Inspection of Takahashi's list}
\label{inspection}
\subsubsection{Excluding K3 and Enriques}
\label{no_k3}
Consider the cases (15) and (16) from Takahashi's list:

(15): $S$ is a K3-surface, $c_1(E)^2 = 10$ or $12$, $c_2(E)= 3$, $\mathrm{det}\, E$ is $2$-jet ample

(16): $S$ is an Enriques surface, $c_1(E)^2 = 12$, $\mathrm{det}\, E$ is $2$-jet ample.

($2$-jet ampleness of $\mathrm{det}\, E$ says nothing in our case because of Theorem \ref{2-jet})

\begin{lemma}
    \label{k3}
    Suppose $S$ is a K3 or Enriques surface and $E$ is as before, $c_2(E) = 3$. Then $h^0(E) = h^0(\mathrm{det}\, E) - 2$.
\end{lemma}
\begin{proof}
    Consider a section $s \in H^0(E)$ with three different zeroes $y_1, y_2, y_3$ and denote by $\eta$ the reduced $0$-dimensional subscheme supported on $\{y_1\} \cup \{y_2\} \cup \{y_3\}$.
    We can build a Koszul complex corresponding to the section $s$:
    \[
    0 \longrightarrow \mathcal O_S \stackrel{\cdot s}{\longrightarrow} E \stackrel {t \mapsto s \wedge t}{\longrightarrow} \mathrm{det}\, E \otimes \mathcal I_{\eta}  \longrightarrow 0
    \]
    Since $H^1(\mathcal O_S) = 0$, the sequence
    \[
    0 \longrightarrow H^0(\mathcal O_S) \longrightarrow H^0(E) \longrightarrow H^0(\mathrm{det}\, E \otimes \mathcal I_{\eta})  \longrightarrow 0
    \]
    is exact too. Hence $h^0(E) = 1 + h^0(\mathrm{det}\, E \otimes \mathcal I_{\eta})$.
    Next we consider the sequence
    \[
    0 \longrightarrow \mathrm{det}\, E \otimes \mathcal I_{\eta} \longrightarrow \mathrm{det}\, E \longrightarrow \mathrm{det}\, E \otimes \mathcal O_{\eta} \longrightarrow 0
    \]
    From Theorem \ref{2-jet} we know that this sequence is exact on the level of global sections, too.
    Therefore we have 
    \[
    h^0(\mathrm{det}\, E \otimes \mathcal I_{\eta}) = h^0(\mathrm{det}\, E) - h^0(\mathrm{det}\, E \otimes \mathcal O_{\eta}) = h^0(\mathrm{det}\, E) - 3
    \]
    and
    \[
    h^0(E) = h^0(\mathrm{det}\, E) - 2
    \]
\end{proof}
By Riemann-Roch we have $h^0(\mathrm{det}\, E) = \chi(\mathcal O_S) + \frac{1}{2} (\mathrm{det}\, E)(\mathrm{det}\, E - K_S) = \chi(\mathcal O_S) + c_1(E)^2/2$ in both cases $(15)$ and $(16)$.
Now consider the case $(15)$: we have $\chi(\mathcal O_S) = 2$ and $h^0(\mathrm{det}\, E) = 2+5$ or $2+6$.
From lemma \ref{k3} we have that $h^0(E) = 5$ or $6$. If $h^0(E) = 5$ then $X$ is a hypersurface in $\mathbb P^4$ and we get a contradiction comparing second Betti number:
one has $b_2(\mathbb P(E)) = b_2(S) + 1 \geq 2$ but the second Betti number of a smooth hypersurface in $\mathbb P^4$ equals one.
For the case $h^0(E) = 6$ we use the classification of three-dimensional scrolls embedded in $\mathbb P^5$ given in \cite{Ottaviani}.
We know that the degree of $X$ is equal to $c_1(E)^2 - c_2(E) = 12 - 3 = 9$, and there is actually one K3-scroll of degree $9$ in this list but it is of the form
$\mathbb P(E')$ for a very ample $E'$ with $c_2(E') = 5$ (and these $S$ and $E'$ are uniquely determined by $X$).
So the case $(15)$ is excluded.

Now consider the case $(16)$. In this case, $\chi(\mathcal O_S) = 1$ and $h^0(E) = 5$, a contradiction.
\subsubsection{Extracting very ample bundles}
\label{extraction}
Let us show that bundles from Theorem \ref{classification3} are very ample and that other bundles from Takahashi's list having $c_2(E) = 3$ are not very ample.
First, let us check the following easy
\begin{prop}
    \label{split}
    A rank-$2$ bundle $E = L \oplus M$, where $L$ and $M$ are line bundles on a surface $S$, is very ample if and only if both $L$ and $M$ are very ample.
\end{prop}
\begin{proof}
    For the `if' part, see \cite[Lemma 3.2.3]{adjunction_theory}.
    Suppose now that $E$ is very ample.
    Arguing by contradiction, suppose that, say, $L$ is not very ample and let us prove that $E$ is not very ample.
    If $L$ does not separate the pair of points $y_1, y_2, \, y_1 \neq y_2;\, y_1, y_2 \in S$, then any section $s \in H^0(X, \mathcal O(1))$ passing through the point $(0\colon 1)$ in the fiber $y_1$ passes through the point $(0\colon 1)$ in the fiber $y_2$ too.
    If $L$ does not separate tangent vectors, i.e., if there exists a point $y\in S$ such that the natural map $H^0(L)\to L \otimes (\mathcal O_{S, y}/\mathfrak m_{S,y}^2)$ is not surjective, then one easily sees that the conditions of Proposition \ref{matrices} do not hold,
    so there exists a point $x \in s_y$ such that
    the natural map $H^0(\mathcal O(1)) \to \mathcal O(1)\otimes (\mathcal O_{X, x})/\mathfrak m_{X, x}^2$ is not surjective.
\end{proof}


So, very ampleness of bundles from the list of Theorem \ref{classification3} in the cases $1, 3, 4, 5$ follows from the above proposition (very ampleness of $C_0 +2f$ on $\mathbb F_1$ follows from \cite[Proposition 5.1]{Beltrametti_kjet}, for example) and very ampleness in the case $2$ is well known.

Now we need to prove that bundles in the cases $(13)$ and $(14)$ of Theorem \ref{Takahashi} are not very ample.
For $(13)$ we need to check that line bundle $C_0 + f$ on $\mathbb P_C(\mathcal E)$ is not very ample and then use Proposition \ref{split} (here $\mathbb P_C(\mathcal E)$ is a $\mathbb P^1$-bundle over an elliptic curve $C$ with invariant $e = -1$).
This statement follows from \cite[Theorem 4.3]{Gushel}.

Consider the remaining case $(14)$.
From the discussion after Claim $3.2$ in \cite{Takahashi} the bundle $E$ is a non-split extension
\[
0\to \mathcal O_S(C_0 +f) \to E \to \mathcal O_S(C_0+f) \to 0
\]
From \cite[Proposition 2.3]{Gushel} we get that $h^1(C_0 + f) = 0$, hence $H^0(E) \cong H^0(C_0+f)\oplus H^0(C_0 +f)$.
Since $C_0 +f$ is not very ample, this decomposition shows that $E$ is not very ample (we can argue as in the proof of Proposition \ref{split}).

\section{Pencils of quadrics}
\label{pencils_of_quadrics}
Let $p \colon X \to C$ be a pencil of quadrics.
According to \cite[Proposition 5.12 and Proposition 5.13]{Lvovski}, if the space of vanishing cycles is 2-dimensional for such an $X$, then all the fibers of $p$ are smooth.
Put $E = p_*\mathcal O(1)$.
$E$ is a rank-$4$ bundle on $C$ and the complete linear system $|\mathcal O_{\mathbb P(E)|S}(1)|$ maps $\mathbb P(E)$ in $\mathbb P^n$ such that the fiber of $p' \colon \mathbb P(E) \to C$ over a point $a \in C$ is mapped isomorphically to the linear span of the quadric $X_a$.
It is clear that there exists a line bundle $\mathcal L$ on $C$ such that $X$ is a zero set of a section $s \in H^0(\mathcal O_{\mathbb P(E)}(2) \otimes (p')^*\mathcal L)$.
\begin{prop}
    Suppose that $p \colon X \to C$ is a pencil of quadrics and $X$ has two-dimensional space of vanishing cycles. Then $C \cong \mathbb P^1$, $\mathcal L \cong \mathcal O_{\mathbb P^1}(-3)$, $E \cong \mathcal O_{\mathbb P^1}(d_1) \oplus \mathcal O_{\mathbb P^1}(d_2) \oplus \mathcal O_{\mathbb P^1}(d_3) \oplus \mathcal O_{\mathbb P^1}(d_4)$, $(d_1, d_2, d_3, d_4) = (1,1,2,2)$,
    and $X$ can be represented as a zero set of a section 
    \[
    s \in H^0(\mathcal O_{\mathbb P(E)}(2)\otimes(p')^*\mathcal L) \cong H^0(\mathrm{Sym}^2 E \otimes \mathcal O_{\mathbb P^1}(-3)),
    \]
    that can be viewed as a $4\times 4$ matrix $a_{ij}$, where $a_{ij} \in H^0(d_i +d_j - 3)$ and the matrix $a_{ij}$ is non-degenerate at every point of $\mathbb P^1$.
    Since $a_{ij}$ is a section of rank-$16$ bundle
    \begin{equation*}
        \begin{pmatrix}
            \mathcal O(-1)&\mathcal O(-1)&\mathcal O    &\mathcal O    \\
            \mathcal O(-1)&\mathcal O(-1)&\mathcal O    &\mathcal O    \\
            \mathcal O   &\mathcal O   &\mathcal O(1)&\mathcal O(1)\\
            \mathcal O   &\mathcal O   &\mathcal O(1)&\mathcal O(1)\\
        \end{pmatrix}
    \end{equation*}
    on $\mathbb P^1$
    non-degeneracy of the matrix $a_{ij}$ at every point in $\mathbb P^1$ is equivalent to the system
    \[
        \left\{
            \begin{array}{rcl}
                a_{13}a_{24}-a_{14}a_{23} \neq 0\\
                a_{13}a_{24}-a_{14}a_{23} \neq 0\\
            \end{array}
        \right.
    \]
    ($a_{13}, a_{24}, a_{14},a_{23},a_{13},a_{24},a_{14},a_{23} \in H^0(\mathcal O) \cong \mathbb C$).
\end{prop}
\begin{proof}
    We imitate the proof of Proposition 5.14 in \cite{Lvovski}.
    Since all the fibers of $p$ are smooth, discriminant of $s$ as a family of quadratic forms, which is a section of $(\mathrm{det}\, E)^{\otimes 2}\otimes L^{\otimes 4}$, has no zeroes on $C$. Hence
    \begin{gather}
        \label{first}
        2 \,\mathrm{deg}\, E + 4\,\mathrm{deg}\, \mathcal L = 0
    \end{gather}
    Now suppose that $Y \subset X$ is a transversal hyperplane section. If $Y$ is the
    zero locus of a section $\sigma \in H^{0}(\mathcal {O}_X(1)) \cong H^0(E)$, then, if $E'$ is the quotient
    in the exact sequence
    \[
    0 \to O_C \stackrel{\sigma}{\to}E \to E' \to 0
    \]
    one sees that $Y \subset \mathbb P(E')$ is the zero locus of a section $\sigma' \in H^0(\mathrm{Sym}^2(E')\otimes L)$
    and that the discriminant of $\sigma'$ is a section of $(\mathrm{det}\, E')^{\otimes 2} \otimes L^{\otimes 3}$;
    since the
    vanishing root system is $A_2$, \cite[Proposition 5.13]{Lvovski} implies that there are precisely
    $3$ degenerate fibers of the induced pencil $Y \to C$, so the degree of the
    invertible sheaf $(\mathrm{det}\,E^{'})^{\otimes 2} \otimes L^{\otimes 3}$ equals 3, whence
    \begin{gather}
        \label{second}
        2\,\mathrm{deg}\, E' +3\,\mathrm{deg}\, \mathcal L = 3
    \end{gather}
    Taking into account that $\mathrm{deg}\, E^{'} = \mathrm{deg}\, E$ and putting together equations \ref{second} and \ref{first},
    one sees that $\mathrm{deg}\, E = 6$ and $\mathrm{deg}\,L = -3$.

    Denote the embedding $i \colon X \hookrightarrow \mathbb P(E)$.
    We have $\mathcal O_{X}(1) = i^*(\mathcal O_{\mathbb P(E)}(2)+(p')^*(L))$ and
    \[
    \mathrm{deg}\, \mathcal O_{X}(1) = 2\,\mathrm{deg}\, (\mathcal O_{\mathbb P(E)}(1)) + \mathrm{deg}\,L = 2\mathrm{deg}\, E + \mathrm{deg}\, L = 12-3=9
    \]
    
    Now we can use the classification of degree nine varieties.
    It follows from \cite{Livorni} that pencils of quadrics (called 'hyperquardic fibrations' in this paper) of degree nine are fibrations over $\mathbb P^1$ (see the case 3.1.3 there).
    Hence $E \cong \mathcal O_{\mathbb P^1}(d_1) \oplus \mathcal O_{\mathbb P^1}(d_2) \oplus \mathcal O_{\mathbb P^1}(d_3) \oplus \mathcal O_{\mathbb P^1}(d_4)$ for some $0\leq d_1 \leq d_2 \leq d_3 \leq d_4$ satisfying $d_1+d_2+d_3+d_4=\mathrm{deg}\, E = 6$.
    The quadratic form defining $X\subset \mathbb P(E)$, which is a section of $\mathrm{Sym}^2\, E \otimes L$, can be represented as a $4\times 4$ matrix $(a_{ij})_{1\leq i, j \leq 4}$, where $a_{ij} \in H^0(d_i+d_j - 3)$.
    The case of $d_1=0$ is impossible. Indeed, if this is the
    case, then $a_{11}$ is identically zero, so in each fiber of the bundle $\mathbb P(E)$ the point
    with homogeneous coordinates $(1 : 0 : 0 : 0)$ lies in $X$ (we use homogeneous
    coordinates that agree with the decomposition $E = \bigoplus \mathcal O_{\mathbb P^1}(d_i)$). On the
    other hand, since $d_1 = 0$, the mapping $\varphi\colon \mathbb P(E) \to \mathbb P^n$ maps the points with
    coordinates $(1 : 0 : 0 : 0)$ in all the fibers of $\mathbb P(E)$ to one and the same point of $\mathbb P^n$.
    Thus, there exists a point contained in all the fibers of the pencil
    $p\colon X \to \mathbb P^1$, which is absurd.

    We have proved that $d_1 \neq 0$, whence $(d_1, d_2, d_3, d_4) = (1, 1, 2, 2)$ or $(1,1,1,3)$.
    But the second case cannot be realised, since any matrix of global sections of
    \begin{equation*}
        A=\begin{pmatrix}
            \mathcal O(-1)&\mathcal O(-1)&\mathcal O(-1)    &\mathcal O(1)   \\
            \mathcal O(-1)&\mathcal O(-1)&\mathcal O(-1)    &\mathcal O(1)  \\
            \mathcal O(-1)  &\mathcal O(-1)  &\mathcal O(-1)&\mathcal O(1)\\
            \mathcal O(1)   &\mathcal O(1)  &\mathcal O(1)&\mathcal O(3)\\
        \end{pmatrix}
    \end{equation*}
    is degenerate at any point $a \in \mathbb P^1$.
    Therefore, the first case holds, i.e., $E \cong \mathcal O_{\mathbb P^1}(d_1) \oplus \mathcal O_{\mathbb P^1}(d_2) \oplus \mathcal O_{\mathbb P^1}(d_3) \oplus \mathcal O_{\mathbb P^1}(d_4)$, $(d_1, d_2, d_3, d_4) = (1,1,2,2)$.
    Note that all matrices of global sections $a_{ij} \in H^0(\mathcal O_{\mathbb P^1}(d_i+d_j-3))$ for which $\mathrm{det}\, (a_{ij})$ is nonzero everywhere, actually define a pencil of quadrics with $2$-dimensional space of vanishing cycles.
    Indeed, since $\mathbb P^1$ is simply connected, the monodromy action of $\pi_1(C)$ on $H_2(X_a, \mathbb Q)$ is trivial; since the condition \ref{second} holds the pencil $Y \to C$ has exactly three degenerate fibers,
    thus by Proposition \cite[Proposition 5.13]{Lvovski} $X$ has $2$-dimensional space of vanishing cycles.
\end{proof}
\bibliography{veryamplepaper}
\bibliographystyle{amsplain}
\end{document}